\documentclass{article}

\usepackage{amsmath,amssymb,amsthm}

\newcommand{\GA}{{\rm GA}}

\newcommand{\EA}{{\rm EA}}

\newcommand{\A}{\mathbb{A}}

\newcommand{\IN}{\mathbb{N}}

\DeclareMathOperator{\Aut}{Aut}

\newtheorem{theorem}{Theorem}

\newtheorem{lemma}[theorem]{Lemma}

\newtheorem{claim}[theorem]{Claim}
\newtheorem{corollary}[theorem]{Corollary}

\theoremstyle{definition}
\newtheorem{definition}{Definition}

\theoremstyle{remark}
\newtheorem{remark}{Remark}
\newtheorem{example}{Example}

\newcommand{\floor}[1]{{\left\lfloor{#1}\right\rfloor}}

\title{A unified approach to embeddings of a line in 3-space}
\author{%
Drew Lewis\thanks{Department of Mathematics and Statistics,  University of South Alabama.  Email address: \texttt{drewlewis@southalabama.edu}}%
}

\begin{document}
\maketitle

\begin{abstract}
While the general question of whether every closed embedding of an affine line in affine \(3\)-space can be rectified remains open, there have been several partial results proved by several different means. We provide a new approach, namely constructing (strongly) residual coordinates, that allows us to give new proofs of all known partial results, and in particular generalize the results of Bhatwadekar-Roy and Kuroda on embeddings of the form \((t^n,t^m,t^l+t)\).
\end{abstract}

\section{Introduction}

Let \(k\) be an algebraically closed field of characteristic zero throughout.  The embedding problem, one of the central problems in affine algebraic geometry, asks whether every closed embedding \(\A^m \hookrightarrow \A^n\) is equivalent to the standard embedding (such embeddings are called {\em rectifiable}).  In general, this is a very difficult problem that has led to several fruitful avenues of research.  If \(n \geq 2m+2\), then every embedding is rectifiable due to a general result of Srinivas \cite{Srinivas}.  Two remaining cases have attracted the most attention: the case where \(n=m+1\), and the case where \(m=1\).  In the former case, it is often referred to as the Abhyankar-Sathaye embedding conjecture, and can be reformulated as asking whether every hyperplane in \(\A^n\) is a coordinate; as this is not the focus of this paper, we simply suggest the papers \cite{Sathaye,RS,WrightCancel,Kaliman} to the interested reader, and note that it remains open for \(n \geq 3\).

In this paper, we are interested in the \(m=1\) case, i.e. embeddings \(\A \hookrightarrow \A^n\). Note that the result of Srinivas \cite{Srinivas} provides an affirmative answer in the case \(n \geq 4\).  In the case \(n=2\), an affirmative answer was provided by Abyhankar and Moh \cite{AM}, and independently, Suzuki \cite{Suzuki}.  Thus, in the \(m=1\) case, the only remaining open case is embeddings \(\A \hookrightarrow \A^3\), so we direct our attention there.

Abhyankar \cite{Abhyankar} conjectured (among other things) that the family of embeddings \((t^n, t^{n+1}, t^{n+2}+t)\) are not rectifiable.  Craighero \cite{Craighero1, Craighero2} showed for \(n=3\) and \(n=4\), these embeddings are, in fact, rectifiable; but it remains an open question for \(n \geq 5\).  Later, Bhatwadekar and Roy \cite{BR} considered the more general family of embeddings \((t^{n}, t^m, t^l+t)\),  and showed that many of these are rectifiable.  More recently, Kuroda \cite{Kuroda} showed that another subset of this family are rectifiable.  These results are stated precisely in Section \ref{sec:unified} below.

All of these papers use different approaches. In \cite{Craighero1}, Craighero explicitly writes down the rectifying automorphism, while in \cite{Craighero2} he explicitly computes a polynomial, and as this polynomial is linear in a variable, appeals to a well-known result of Sathaye \cite{Sathaye} to see it can be extended to a rectifying automorphism.  Bhatwadekar and Roy \cite{BR} make use of the Dedekind conductor, while Kuroda \cite{Kuroda} constructs rectifying automorphisms from exponentials of locally nilpotent derivations similar to the Nagata map.   

In this paper, we take a single approach that we apply to recover all of these results and extend some of them.  To broadly explain our approach, first note that an embedding \(\mathbb{A} \hookrightarrow \mathbb{A}^3\) corresponds to a surjection of polynomial rings \(k[x,y,z] \rightarrow k[t]\).  It is well-known that an embedding is rectifiable if and only if there is a coordinate of \(k[x,y,z]\) that maps to \(t\).  Our approach is to construct (strongly) residual coordinates that map to \(t\), and then use a criterion of \cite{StronglyResidual} to see that they are coordinates (see Corollary \ref{cor:criterion} below).

In the next section, we review some standard definitions and results from the field, and then describe our general approach.  Then in Section \ref{sec:unified}, we precisely state and give new proofs of the results of \cite{Craighero1, Craighero2, BR, Kuroda}; in particular, we also prove generalizations of the results of Bhatwadekar and Roy \cite{BR} and Kuroda \cite{Kuroda} in Theorems \ref{thm:BRgeneral} and \ref{thm:Kurodageneral}, respectively.

\section{Preliminaries}\label{sec:preliminaries}

Let us begin by quickly recalling some fairly standard definitions; more detailed explanations can be found in \cite{Arno} or \cite{Wright}.
\begin{itemize}
\item We use \(\GA_n(k)\) to denote the general automorphism group \(\Aut _k \mathbb{A} ^n\).  This group is anti-isomorphic to the automorphism group of the polynomial ring \(k^{[n]}\).  
\item  \(\EA_n(k)\) denotes the subgroup of \(\GA_n(k)\) generated by elementary automorphisms, i.e. those fixing \(n-1\) variables. 
\item A polynomial \(f \in k^{[n]}\) is called a {\em coordinate} (or {\em variable} by some authors) if there exist \(f_2,\ldots,f_n\) such that \((f,f_2,\ldots,f_n) \in \GA_n(k)\). 
\item An embedding \(\phi:\A \hookrightarrow \A^3\) is called {\em rectifiable} if there exists \(\theta \in \GA_3(k)\) such that \(\theta \phi = (t,0,0)\).
\item For any morphism \(\phi:\A^m \rightarrow \A^n\), we use \(\phi ^*\) to denote the corresponding map \(\phi^*:k^{[n]} \rightarrow k^{[m]}\).
\end{itemize}

 The following is a well-known algebraic characterization of rectifiable embeddings.
\begin{lemma}\label{lem:coordinate}
Let \(\phi: \A \hookrightarrow \A^n\) be an embedding.  If \(f \in k^{[n]}\) is a coordinate with \(\phi^*(f)=t\), then \(\phi\) is rectifiable.
\end{lemma}
\begin{proof}
Let \(\theta=(f,f_2,\ldots,f_n) \in \GA_n(k)\).  Since \(\phi^*(f)=t\), we have \(\theta \phi=(t,g_2(t),\ldots,g_n(t))\) for some polynomials \(g_2, \ldots, g_n \in k[t]\).  Let \(\psi = (x_1,x_2-g_2(x_1),\ldots,x_n-g_n(x_1)) \in \GA_n(k)\).  Then \(\psi \theta \phi = (t,0,\ldots,0)\), so \(\phi\) is rectifiable.
\end{proof}

This lemma underlies our approach in this paper: given an embedding, we aim to construct a coordinate that the embedding maps to \(t\).  To construct coordinates, we rely up on the idea of (strongly) residual coordinates.

\subsection{Strongly residual coordinates and associated embeddings}

\begin{definition}A polynomial \(f \in k[x]^{[n]}\) is called a {\em residual coordinate} if it becomes a coordinate modulo \(x-c\) for each \(c \in k\).  A polynomial is called a {\em strongly residual coordinate} if it becomes a coordinate modulo \(x\), and also a coordinate over \(k[x,x^{-1}]\).
\end{definition}

Note that strongly residual coordinates are residual coordinates. A special case of the Dolgachev-Weisfeiler conjecture asserts that all residual coordinates are coordinates.  This remains open for \(n \geq 3\), with the V\'en\'ereau polynomial \(y+x(xz-y(yu+z^2))\) being a famous example of a (strongly) residual coordinate whose status as a coordinate remains unresolved\footnotemark.

\footnotetext{In fact, the V\'en\'ereau polynomial is also a hyperplane, meaning it is a potential counterexample to the Abhyankar-Sathaye conjecture as well. We refer the reader to \cite{KVZ,Vtype,bivariables} for further reading on the V\'en\'ereau polynomial.}

To illustrate our methods, let us consider the well known construction of the Nagata automorphism.  Let \(\alpha = (x,y,z-\frac{y^2}{x}) \in \GA_2(k[x,x^{-1}])\), and set \(\beta = (x,y+x^2z,z) \in \GA_2(k[x])\).  It is straightforward to compute that \[\alpha ^{-1} \beta \alpha = (x,y+x(xz-y^2), z+2y(xz-y^2)+x(xz-y^2)^2).\]
This resulting element of \(\GA_2(k[x])\) is known as the Nagata automorphism; after many years, it was famously shown \cite{SU} to not be generated by elementary and affine elements of \(\GA_3(k)\).  Since it is a conjugation of an elementary automorphism (over \(k[x,x^{-1}]\)), it can also be obtained as an exponential of a locally nilpotent derivation, namely \((xz-y^2)\left(x\frac{\partial}{\partial y}+2y\frac{\partial}{\partial z}\right)\); Kuroda exploits this fact in \cite{Kuroda}.  However, if we instead consider \(\alpha _0 = (x,y,z-\frac{y^2}{x^2})\) and \(\beta _0 = (x,y+x^3z,z)\), then \(\alpha _0 ^{-1} \beta_0 \alpha _0 \notin \GA_2(k[x])\); however, letting \(\gamma = (x,y,z+\frac{2y^3}{x})\), then a straightforward computation verifies that \(\gamma \alpha _0 ^{-1} \beta _0 \alpha _0 \in \GA_2(k[x])\).  Note that since we are no longer simply conjugating, it is not an exponential of a locally nilpotent derivation.

We now generalize this approach; the following is a special case of Theorem 13 of \cite{StronglyResidual}, but we also present a short direct proof here.
\begin{theorem}\label{thm:cleardenominators}
Let \(\alpha, \beta \in \EA_2(k[x,x^{-1}])\) be of the form \(\alpha =(x,y,z+x^{-l}P(y))\) for some \(P \in k[x][y]\) and \(\beta = (x,y+x^a(x^lz)^b,z)\) for some \(a,b \in \IN\). Then there exists \(R \in k[x][y]\) such that setting \(\gamma = (x,y,z+x^{-l}R(y))\), then \(\gamma \beta \alpha \in \GA_2(k[x])\).  Moreover, if \(P(y) \in (y^c)\) for some \(c \in \IN\), then \(R(y)\in (y^c)\).
\end{theorem}

\begin{proof}
We begin by proving the following:
\begin{claim}
Let \(F \in k[x]^{[2]}\) and \(G \in k[x]^{[1]}\), and suppose \[\phi = (x,y+x^a(x^lz+P(y))^b,z+x^azF(x^lz,y)+x^{-l}G(y)) \in \GA_2(k[x,x^{-1}]).\]  Then setting \(\theta = (x,y,z-x^{-l}G(y)) \in \EA_2(k[x,x^{-1}])\),  there exist \(\tilde{F} \in k[x]^{[2]}\) and \(\tilde{G} \in k[x]^{[1]}\) such that 
\[\theta \phi = (x,y+x^a(x^lz+P(y))^b,z+x^az\tilde{F}(x^lz,y)+x^{-l+a}\tilde{G}(y)) \in \GA_2(k[x,x^{-1}]).\]
Moreover,  \(\tilde{G}(y) \in (y^c)k[x,y]\) as well.
\end{claim}
\begin{proof}
This follows from direct computation. Note that \(\theta\) fixes \(x\) and \(y\), so we need only compute 
\[(\theta \phi)^*(z) = z+x^azF(x^lz,y)+x^{-l}\left(G(y) - G(y+x^a(x^lz+P(y))^b)\right).\]
The key observation is to apply Taylor's formula to obtain 
\[G(y)-G(y+x^a(x^lz+P(y))^b) = -\sum _{i=1} \frac{1}{i!}G^{(i)}(y)(x^a(x^lz+P(y))^b)^i.\] 
Applying a binomial expansion to \((x^lz+P(y))^{bi}\), we see that there exists \(\tilde{F}_0 \in k[x]^{[2]}\) such that
\[G(y)-G(y+x^a(x^lz+P(y))^b) = -\sum _{i=1} \frac{1}{i!}G^{(i)}(y)x^{ai}(P(y))^{bi} + x^{l+a}z\tilde{F}_0(x^lz,y). \]
Then we observe that setting 
\begin{align*}
\tilde{G}(y)=-\sum _{i=1} \frac{1}{i!}G^{(i)}(y)x^{a(i-1)}(P(y))^{bi} &&\text{and}&&  \tilde{F}(x^lz,y)=F(x^lz,y)+\tilde{F}_0(x^lz,y),\end{align*}
we obtain the desired result.

\end{proof}

Now, beginning with \(\beta \alpha\) and repeatedly applying the claim, we may inductively produce a sequence \(\theta _0, \ldots, \theta _r\) with \(\theta _i = (x,y,z-x^{-l+ai}G_i(y))\) (with each \(G_i(y) \in (y^c)\)) such that 
\[\theta _r \cdots \theta _0 \beta \alpha = (x,y+x^a(x^lz+P(y))^b,z+x^az\tilde{F_r}(x^lz,y)+x^{-l+a(r+1)}\tilde{G_r}(y))\]
for some \(\tilde{F_r}(x^lz,y), \tilde{G_r}(x^lz,y) \in k[x][x^lz,y]\).  Let \(r\) be minimal so that \(a(r+1) \geq l\), and set \(\gamma = \theta _r \cdots \theta _0\).  Then letting \(R(y) = -\sum _{i=0} ^r x^{ai}G_i(y) \in (y^c)\), we see \(\gamma = (x,y,z+x^{-l}R(y))\) with \(\gamma \beta \alpha \in \GA_2(k[x,x^{-1}])\) and \(\gamma \beta \alpha \in (k[x,y,z])^3\).  It is well known (e.g. Proposition 1.1.7 in \cite{Arno}) that this implies that \(\gamma \beta \alpha \in \GA_2(k[x])\) as required.
\end{proof}

\begin{corollary}\label{cor:criterion}
Let \(\phi: \A \hookrightarrow \A^3\) be an embedding.  Suppose \(\alpha, \beta \in \EA_2(k[x,x^{-1}])\) are of the form \(\alpha =(x,y,z+x^{-l}P(y))\) and \(\beta = (x,y+xQ(x^lz),z)\) for some \(P,Q \in k[x]^{[1]}\), and that \(\phi^* \alpha^* \beta^*(y)=t\).  Then \(\phi\) is rectifiable.
\end{corollary}
\begin{proof}
Apply the previous theorem to produce \(\gamma = (x,y,z+x^{-l}R(y))\) such that \(\gamma \beta \alpha \in \GA_2(k[x])\). Let \(f=(\gamma \beta \alpha)^*(y)\), so \(f\) is a coordinate.  Since \(\gamma^*(y)=y\), we have \(t=\phi^*\alpha^*\beta^*(y) = \phi^*\alpha^*\beta^*\gamma^*(y)=\phi^*(\gamma \beta \alpha)^*(y)=\phi^*(f)\).  Lemma \ref{lem:coordinate} then completes the proof.
\end{proof}

\begin{remark}
An elementary example of how this corollary can be used can be found in Theorem \ref{thm:Craighero1}.
\end{remark}
\subsection{A useful combinatorial lemma}
Here we prove a combinatorial lemma that we use in the proof of Lemma \ref{lem:cd} below; the reader may prefer to skip the proof for now, and proceed to Section \ref{sec:unified}. This is used only in the proof of Lemma \ref{lem:cd} below.
For an integer \(m\), let us write \(\delta _m = \begin{cases} 1 & m \text{ is even} \\ 0 & m \text{ is odd}\end{cases}\).
\begin{lemma}\label{lem:combinatorial}
Let \(m \in \IN\).  Then there exist \(\alpha _0, \ldots, \alpha _\floor{\frac{m}{2}}, \beta \in k\) such that, in \(k[s]\),
\[ \sum _{i=0} ^\floor{ \frac{m}{2}}  \alpha _i s^i(1+s)^{m-2i} = 1+ \beta \delta _m s^{\frac{m}{2}}+s^m .\]
\end{lemma}

\begin{proof}
First, suppose \(m=2k+1\) is odd. We induct on \(k\), with the \(k=0\) case being trivial. For \(k>0\), we first expand \((1+s)^{2k+1}\) noting the symmetry of coefficients:

\begin{equation*}
(1+s)^{2k+1} = 1+s^{2k+1} + \sum _{j=1} ^{k} {2k+1 \choose j} (s^j + s^{2k+1-j}) 
= 1+ s^{2k+1} + \sum _{j=1} ^{k} {2k+1 \choose j}s^j (1 + s^{2k+1-2j}). 
\end{equation*}

For \(1 \leq j \leq k\), by the induction hypothesis we choose \(\alpha _{j,0},\ldots,\alpha _{k-j}\) such that 
\[1+s^{2k+1-2j} =\sum _{i=0} ^{k-j} \alpha _{j,i} s^i (1+s)^{2k+1-2j-2i}.\]
Then
\begin{align*}
(1+s)^{2k+1} &=
 1+ s^{2k+1} + \sum _{j=1} ^{k} {2k+1 \choose j}s^j \sum _{i=0} ^{k-j} \alpha _{j,i} s^i (1+s)^{2k+1-2j-2i}.
\end{align*}
Letting \(\alpha _{0,i} =  -{2k+1 \choose i}\) for \(0 \leq i \leq k\), we then have
\begin{align*} 
1+s^{2k+1}&= -\sum _{j=0} ^{k} \sum _{i=0} ^{k-j}{2k+1 \choose j} \alpha _{j,i} s^{i+j} (1+s)^{2k+1-2(i+j)}  \\
&=  \sum _{r=0} ^k  s^{r} (1+s)^{2k+1-2r} \left(-\sum _{j=0} ^{r} {2k+1 \choose j} \alpha _{j,r-j}\right) .
\end{align*}

Now, suppose \(m=2k\) is even. We again induct on \(k\), with \(k=0\) being trivial.  For \(k>0\) we proceed similarly by expanding \((1+s)^{2k+1}\), noting however that there is an odd number of terms now.  

\begin{equation*}
(1+s)^{2k} = 1+{2k \choose k } s^{k}+s^{2k} + \sum _{j=1} ^{k-1} {2k \choose j} (s^j + s^{2k-j}) 
= 1+ {2k \choose k} s^{k}+s^{2k} + \sum _{j=1} ^{k-1} {2k \choose j}s^j (1 + s^{2k-2j}) 
\end{equation*}

For \(1 \leq j \leq k-1\), by the induction hypothesis we choose \(\alpha _{j,0},\ldots,\alpha _{k-j}, \beta _j\) such that 
\[1+s^{2k-2j} = \beta _j s^{k-j}+\sum _{i=0} ^{k-j} \alpha _{j,i} s^i (1+s)^{2k-2j-2i}.\]
Then
\begin{align*}
(1+s)^{2k} &=
 1+ {2k \choose k} s^{k}+s^{2k} + \sum _{j=1} ^{k-1} {2k \choose j}s^j \left( \beta _j   s^{k-j}+\sum _{i=0} ^{k-j} \alpha _{j,i} s^i (1+s)^{2k-2j-2i} \right) \\
&=1+ s^{k} \left({2k \choose k}+\sum _{j=1} ^{k-1} {2k \choose j} \beta _j\right)+s^{2k} +\sum _{j=1} ^{k-1} \sum _{i=0} ^{k-j}{2k \choose j} \alpha _{j,i} s^{i+j} (1+s)^{2k-2(i+j)} 
\end{align*}

Now, set  \(\alpha _{0,i} = -{2k \choose i}\) for \(0 \leq i \leq k\), and set \(\beta = {2k \choose k}+\sum _{j=1} ^{k-1} {2k \choose j} \beta _j\).  Then

\begin{align*}
1+\beta s^k + s^{2k} &= 
-\sum _{j=0} ^{k} \sum _{i=0} ^{k-j}{2k \choose j} \alpha _{j,i} s^{i+j} (1+s)^{2k-2(i+j)}\\ 
 &=\sum _{r=0} ^k  s^{r} (1+s)^{2k-2r} \left(-\sum _{j=0} ^{r} {2k \choose j} \alpha _{j,r-j}\right) .
\end{align*}

\end{proof}

\section{A unified approach to known results on embedding}\label{sec:unified}

In this section we summarize all results known to us on embeddings \(\A \hookrightarrow \A^3\), and show how they can all be proved via Corollary \ref{cor:criterion}.  For convenience, we will adopt the notation \(X=\phi^*(x)\), \(Y=\phi^*(y)\), and \(Z=\phi^*(z)\).  So for example, when considering the Abhyankar embeddings \((t^n, t^{n+1}, t^{n+2}+t)\), we will write \(X=t^n\), \(Y=t^{n+1}\), and \(Z=t^{n+2}+t\).  As an elementary example, consider the \(n=2\) Abhyankar embedding \((t^2, t^3, t^4+t)\).  It is straightforward to see that \(Z-X^2=t\), and the polynomial \(z-x^2\) is a coordinate (of the elementary automorphism \((x,y,z-x^2)\)), so by Lemma \ref{lem:coordinate} the embedding \((t^2,t^3,t^4+t)\) is rectifiable.

\subsection{Craighero's results}
The \(n=3\) and \(n=4\) cases of Abhyankar's conjecture were resolved by Craighero \cite{Craighero1,Craighero2}; we are able to give short proofs here. 

\begin{theorem}[Craighero]\label{thm:Craighero1}
The embedding \(\phi = (t^3,t^4,t^5+t)\) is rectifiable.
\end{theorem}
\begin{proof}
We begin by computing 
\[Y-\frac{Z^2}{X^2}+2 = t^4-\frac{t^{10}+2t^6+t^2}{t^6} +2 =-t^{-4}.\]
Let \(\alpha = \left(x,y-\frac{z^2}{x^2}+2,z\right)\) and \(\beta = (x,y,z+x(x^2y))\). Then we compute 
\begin{align*}
\phi^*\alpha^*\beta^*(z) &= \phi^*\alpha^*(z+x(x^2y)) \\
&= \phi^*\left(z+x^3\left(y-\frac{z^2}{x^2}+2\right)\right) \\
&= Z+X^3\left(Y-\frac{Z^2}{X^2}+2\right) \\
&= (t^5+t)+t^{9}(-t^{-4}) \\
&= t.
\end{align*}
Thus, by Corollary \ref{cor:criterion}, the embedding is rectifiable.  
\end{proof}

\begin{theorem}[Craighero]\label{thm:Craighero2}
The embedding \(\phi = (t^4,t^5,t^6+t)\) is rectifiable.
\end{theorem}
\begin{proof}
We first observe that \(Z^2-X^3=2t^7+t^2\). Next, letting \(a,b \in k\) be arbitrary for the moment, we compute
\begin{align*}
Y+a\frac{Z^3}{X^2}+bXZ -3a &= t^5+\frac{a\left(t^{18}+3t^{13}+3t^8+t^3\right)}{t^8}+b(t^{10}+t^5) -3a \\
&= (a+b)t^{10}+(1+3a+b)t^5+at^{-5}.
\end{align*}
Setting \(a=-\frac{1}{2}\) and \(b=\frac{1}{2}\) causes these first two coefficients to vanish, so we see
\[Y-\frac{1}{2}\frac{Z^3}{X^2}+\frac{1}{2}XZ+\frac{3}{2} = -\frac{1}{2}t^{-5}. \]
Thus,
\[\left(Y-\frac{1}{2}\frac{Z^3}{X^2}+\frac{1}{2}XZ+\frac{3}{2} \right) +\frac{1}{4}\frac{Z^2-X^3}{X^3}=-\frac{1}{2}t^{-5}+\frac{1}{4}\frac{2t^7+t^2}{t^{12}} =\frac{1}{4}t^{-10}.\]
Now, let 
\begin{align*}
\alpha &= \left(x, y-\frac{1}{2}\frac{z^3}{x^2}+\frac{1}{2}xz+\frac{3}{2}  +\frac{1}{4}\frac{z^2-x^3}{x^3},z\right) \\
\beta &= (x,y,z+x(-4x^3y)) 
\end{align*}
Then we see 
\begin{align*}
\phi^*\alpha^*\beta^*(z) &= Z-4X^4\left(Y-\frac{1}{2}\frac{Z^3}{X^2}+\frac{1}{2}XZ+\frac{3}{2}  +\frac{1}{4}\frac{Z^2-X^3}{X^3}\right) \\
&= t^{6}+t-4(t^{16})\left(-\frac{1}{4}t^{-10}\right) \\
&= t.
\end{align*}
Thus \(\phi\) is rectifiable by Corollary \ref{cor:criterion}.

\end{proof}

\subsection{Generalizing Bhatwadekar-Roy's results}
The next results were due to Bhatwadekar and Roy \cite{BR}, who studied embeddings of the form \((t^n, t^{an+1}, t+t^l)\) for positive integers \(a,n\) and \(l>n\).  This more general class includes the Abhyankar embeddings mentioned above.  Interstingly, in positive characteristic they were able to show that all such embeddings are rectifiable; in the characteristic zero case we are interested in, they obtained two positive results: when \(l\equiv -1 \pmod n\), and when \(n=4\).  We first generalize the former in Theorem \ref{thm:BRgeneral}, and then provide a new proof of the latter in Theorem \ref{thm:BR2}.

\begin{theorem}\label{thm:BRgeneral}
Let \(m,n \in \IN\) be coprime positive integers. Write \(m \equiv c \pmod n\) for some \(0<c<n\), and set \(d=n-c\). Write \(\lambda_1 c=\mu_1 n -1\) and \(\lambda _2 d = \mu _2 n -1\) for \(\lambda _1, \lambda _2 \in \IN\) and minimal \(\mu_1, \mu_2 \in \IN\). If \(b > \min\{\mu_1, \mu_2\}\), then  the embedding \(\phi = (t^n, t^{m}, t+t^{bn-1})\) is rectifiable.
\end{theorem}  

\begin{remark}
When \(c=1\), we obtain one of the results of Bhatwadekar and Roy in \cite{BR}.
\end{remark}

\begin{example}
Let \(n=5\) and \(m=7\), so \(c=2\), in which case \(\mu_1=1\).  Then we see the embedding \((t^5,t^7,t+t^9)\) is rectifiable.
\end{example}

More generally, we have
\begin{corollary}
Let \(n \in \IN\) be odd.  Then for any \(b>1\), the embedding \((t^n,t^{n+2}, t+t^{bn-1})\) is rectifiable.
\end{corollary}
\begin{proof}
Note that \(c=2\), so since \(n\) is odd, \(n-1\) is even and thus \(\mu_1=1\).
\end{proof}

%
In order to prove Theorem \ref{thm:BRgeneral}, we first prove the following lemma.
\begin{lemma}\label{lem:cd} Suppose \(\phi=(t^n, t^m, t^{bn-1})\) is an embedding.  Then for any \(r \in \IN\) there exits \(p \in k[x,z]\) such that \(\phi^*(p)=t^r+t^{(bn-1)r}\).
\end{lemma}
\begin{proof}
We begin by applying Lemma \ref{lem:combinatorial}.  Recall that \(\delta _r = \begin{cases} 1 & r \text{ is even} \\ 0 & r \text{ is odd}\end{cases}\); then by Lemma \ref{lem:combinatorial}, there exist \(\alpha _0, \ldots, \alpha _{ \floor{\frac{r}{2}}}, \beta \in k\) such that, for any \(s \in k[t]\),
\begin{equation}\label{eq:s}
 \sum _{i=0} ^\floor{ \frac{r}{2}}  \alpha _i s^i(1+s)^{r-2i} = 1+ \beta \delta _{r} s^{\frac{r}{2}}+s^{r} .
 \end{equation}
Now we set \(\displaystyle p=  \sum _{i=0} ^\floor{\frac{r}{2}} \alpha _i x^{bi}z^{r-2i} -\beta \delta _{r} x^{b \frac{r}{2}} \in k[x,z]\), and we compute
\begin{align*}
\phi^*(p)&= \sum _{i=0} ^\floor{\frac{r}{2}} \alpha _i (t^n)^{bi}(t+t^{bn-1})^{r-2i} - \beta \delta _{r} (t^n)^{b \frac{r}{2}}  \\
&= \sum _{i=0} ^\floor{\frac{r}{2}} \alpha _i t^{nbi+r-2i}(1+t^{bn-2})^{r-2i} - \beta \delta _{r} t^{bn \frac{r}{2}} \\
&= t^{r}\left(\sum _{i=0} ^\floor{\frac{r}{2}} \alpha _i \left(t^{bn-2}\right)^i(1+t^{bn-2})^{r-2i} - \beta \delta _{r} \left(t^{bn-2}\right)^{\frac{r}{2}}\right)
\end{align*}
Substituting \(s=t^{bn-2}\) into \eqref{eq:s} above, we obtain
\[\phi^*(p)=t^{r}\left(1+\left(t^{bn-2}\right)^{r}\right)=t^{r}+t^{(bn-1)r}.\]
\end{proof}

\begin{proof}[Proof of Theorem \ref{thm:BRgeneral}]
Write \(m=an+c\) for some \(a \in \IN\).  We begin by supposing \(a \geq bd-1\), so \(a+1=bd+h\) for some \(h \in \IN\).  In this case, we can write \(m=(a+1)n-d=(bn-1)d+hn\).  By Lemma \ref{lem:cd}, there exists \(p \in k[x,z]\) such that \(\phi^*(p)=t^d+t^{(bn-1)d}\).  Then, setting \(\theta = (x,-y+x^hp(x,z),z) \in \GA_3(k)\), we see that 
\(\theta \phi = (t^n, t^{hn+d},t+t^{bn-1})\), and this is rectifiable if and only if \(\phi\) is rectifiable. Note that since \(b\geq 2\), \(hn+d < (bn-1)d+hn=m\).  Thus, by repeating this process we may continue until we have \(a < bd-1\). 

So now it suffices to assume \(a <bd-1\).  We proceed in two cases; first, we suppose \(b > \mu _2\). By Lemma \ref{lem:cd}, there exists \(p \in k[x,z]\) such that \(\phi^*(p)=t^d+t^{(bn-1)d}\).
We then set \(\alpha = \left(x,y-\frac{p(x,z)}{x^{bd-a-1}},z\right) \) and \(\beta = \left( x,y,z+ x^{b-\mu _2}\left(x^{bd-a-1}y\right)^{\lambda _2} \right)\), and compute
%
%
%
\begin{align*}
\phi^* \alpha^* \beta^*(z) &=  Z+ X^{b-\mu _2}(X^{bd-a-1}Y-p(X,Z))^{\lambda _2} \\
&= t+t^{bn-1} + \left(t^n\right)^{b-\mu _2} \left( (t^n)^{bd-a-1} t^{an+c} - \left( t^{d}+t^{(bn-1)d}\right) \right)^{\lambda _2} \\
&=t+t^{bn-1}+t^{n(b-\mu _2)}\left(t^{bdn-n+c} - t^{d}-t^{(bn-1)d} \right)^{\lambda _2} \\
&= t+t^{bn-1} +t^{nb-\lambda _2 d - 1}(-t^{\lambda _2 d}) \\
&= t.
\end{align*}
Then by Corollary \ref{cor:criterion}, \(\phi\) is rectifiable.

Now, we must consider the case \(b > \mu _1\).  Then by Lemma \ref{lem:cd}, find \(q \in k[x,z]\) such that \(\phi^*(q)=t^c+t^{(bn-1)c}\), and as above, \(p \in k[x,z]\) such that \(\phi^*(p)=t^d+t^{(bn-1)d}\).

Now set \(\alpha = \left(x,y-\frac{p(x,z)}{x^{bd-a-1}}+\frac{q}{x^{bc+bd-a-2}},z\right) \) and \(\beta = \left( x,y,z- x^{b-\mu _1}\left(x^{bc+bd-a-2}y\right)^{\lambda _1} \right)\), and compute
\begin{align*}
\phi^* \alpha^* \beta^*(z) &=  Z- X^{b-\mu _1}(X^{bc+bd-a-2}Y-X^{bc-1}p(X,Z)+q(X,Z))^{\lambda _1} \\
&= t+t^{bn-1} - \left(t^n\right)^{b-\mu _1} \left( (t^n)^{bc+bd-a-2} t^{an+c} -(t^n)^{bc-1} \left( t^{d}+t^{(bn-1)d}\right)+\left(t^c+t^{(bn-1)c}\right) \right)^{\lambda _1} \\
&=t+t^{bn-1}-t^{n(b-\mu _1)}\left(t^{bcn+bdn-2n+c} - t^{bcn-n+d}-t^{bnc-n+(bn-1)d}+t^c+t^{(bn-1)c} \right)^{\lambda _1} \\
&= t+t^{bn-1} -t^{nb-\lambda _1 c - 1}(t^{\lambda _1 c}) \\
&= t.
\end{align*}
Then again by Corollary \ref{cor:criterion}, \(\phi\) is rectifiable.
\end{proof}

We next give a new proof of the other result of \cite{BR}.
\begin{theorem}[Bhatwadekar-Roy] \label{thm:BR2}
\(\phi = (t^4, t^{4a+1}, t^m+t)\) is rectifiable for any \(a,m \in \IN\).
\end{theorem}
\begin{proof}
We divide the proof into four cases, based on the residue of \(m\) modulo \(4\).  Three cases are relatively straightforward, while the case \(m \equiv 2 \pmod 4\) generalizes our proof of Theorem \ref{thm:Craighero2} above.

\noindent \textbf{Case 1: \(m \equiv 1 \pmod 4\).} Write \(m=4k+1\) for some \(k \in \IN\).  Without loss of generality, we may assume \(a < k\); for if \(a \geq k\), we may apply the map \((x,-y+x^{a-k}z,z)\) to produce the embedding \((t^4, t^{4(a-k)+1},t^m+t)\), and repeat until \(a <k\).  But in this case, observe that \(Z-X^{k-a}Y=t\), and \(z-x^{k-a}y\) is a coordinate.  

\noindent \textbf{Case 2: \(m \equiv 2 \pmod 4\).} Write \(m=4k+2\) for some \(k \in \IN\).  This is the hardest case, but proceeds in the same way as our proof of Theorem \ref{thm:Craighero2}.  First, we observe that
\[Z^2-X^{2k+1}=2t^{4k+3}+t^2.\]

\begin{align*}
Y+c\frac{Z^3}{X^{2k+1-a}}+dX^{a}Z-3cX^{a-k} &=(c+d)t^{4a+4k+2}+(1+3c+d)t^{4a+1}+c\frac{1}{t^{8k-4a+1}}
\end{align*}
Setting \(c=-\frac{1}{2}\) and \(d=\frac{1}{2}\) causes these first two coefficients to vanish, so we see
\[Y-\frac{1}{2}\frac{Z^3}{X^{2k+1-a}}+\frac{1}{2}X^aZ+\frac{3}{2}X^{a-k} = -\frac{1}{2}\frac{1}{t^{8k-4a+1}} \]

Thus
\[Y-\frac{1}{2}\frac{Z^3}{X^{2k+1-a}}+\frac{1}{2}X^aZ+\frac{3}{2}X^{a-k} +\frac{1}{4} \frac{Z^2-X^{2k+1}}{X^{3k-a+1}} = \frac{1}{4}\frac{1}{t^{12k-4a+2}} \]

Now, let 
\begin{align*}
\alpha &= \left(x, \left(y-\frac{1}{2}\frac{z^3}{x^{2k+1-a}}+\frac{1}{2}x^az+\frac{3}{2}x^{a-k} \right) +\frac{1}{4}\frac{z^2-x^3}{x^{3k-a+1}},z\right) 
\end{align*}

Note that if \(a > 3k+1\), then \(\alpha \in \EA_3(k)\).  And in this case, we have \(12k-4a+2=4(3k-a+1)-2<0\); so then we have \(\alpha \phi= (t^4, \frac{1}{4}t^{4a-(12k+2)},t^{m}+t) \), with \(4a-(12k+2)< 4a+1\). Now, if \( k\geq a-3k-1\), we can set \(\beta = (x,y,z-4x^{4k-a+1}y) \in \EA_3(k)\) , and compute \( \phi^*\alpha^*\beta^*(z) = t\), and as \(\beta \alpha \in \GA_3(k)\), \(\phi\) is rectifiable by Lemma \ref{lem:coordinate}.  If instead \(k < a-3k-1\), we set \(\beta = (x,-4y+x^{a-4k-1}z,z) \in \GA_3(k)\), and compute that \(\beta \alpha \phi = (t^4, t^{4(a-4k-1)+1},t^m+t)\).  

This process can be repeated, so we may now proceed assuming without loss of generality that \(a \leq 3k+1\).  We construct \(\alpha\) as above (but now, \(\alpha \in \EA_2(k[x,x^{-1}])\)), and further define 
\begin{align*}
\beta &=  \left(x,y,z+x^{k}(-4x^{3k-a+1}y)\right)
\end{align*}
Then we see 
\begin{align*}
\phi^*\alpha^*\beta^*(z) &= Z-4X^{4k-a+1}\left(Y-\frac{1}{2}\frac{Z^3}{X^{2k+1-a}}+\frac{1}{2}X^aZ+\frac{3}{2}X^{a-k}  +\frac{1}{4}\frac{Z^2-X^{2k+1}}{X^{3k-a+1}}\right) \\
&= t^{4k+2}+t-4(t^{16k-4a+4})\left(\frac{1}{4}t^{-12k+4a-2}\right) \\
&= t.
\end{align*}
Thus \(\phi\) is rectifiable by Corollary \ref{cor:criterion}.

%
%

\noindent \textbf{Case 3: \(m \equiv 3 \pmod 4\).} Write \(m=4k+3\) for some \(k \in \IN\). Without loss of generality, we may assume \(a \leq 3k+2\); for if \(a >3k+2\), we set \(\alpha = (x,-y+x^{a-2-3k}z^3-3x^{a-2k-1}z,z) \in \GA_3(k)\) and compute
\begin{align*}
\phi^*\alpha^*(y) &= -Y+X^{a-2-3k}Z^3-3X^{a-2k-1}Z \\
&=-t^{4a+1}+t^{4(a-2-3k)}(t^{4k+3}+t)^3 -3t^{4a-8k-4}(t^{4k+3}+t)\\
&= -t^{4a+1}+t^{4a-8-12k}\left(t^{12k+9}+3t^{8k+7}+3t^{4k+5}+t^3\right) - 3t^{4a-4k-1}-3t^{4a-8k-3} \\
&= t^{4a-12k-5}
\end{align*}
This process can  be repeated until \(a \leq 3k+2\).

Now, assuming \(a \leq 3k+2\), we set
\(\alpha = \left(x,y-\frac{z^3}{x^{3k+2-a}}-3x^{a-2k-1}z,z\right)\) and \(\beta = (x,y,z+x^k(-x^{3k+2-a}y))\),
and compute
\begin{align*}
\phi^*\alpha^*\beta^*(z) &= Z-X^K\left(Z^3-X^{3k-a+2}Y-3X^{k+1}Z\right)  \\
&=t+t^{4k+3}-t^{4k}\left( (t^{4k+3}+t)^3-t^{4(3k-a+2)}(t^{4a+1})-3t^{4k+4}(t^{4k+3}+t)\right) \\
&= t+t^{4k+3}-t^{4k}\left(t^{12k+9}+3t^{8k+7}+3t^{4k+5}+t^3 - t^{12k+9}-3t^{8k+7}-3t^{4k+5}\right)  \\
&= t+t^{4k+3}-t^{4k}\left(t^3 \right) \\
&= t.
\end{align*}
Thus \(\phi\) is rectifiable by Corollary \ref{cor:criterion}.

\noindent \textbf{Case 4: \(m \equiv 0 \pmod 4\).} In this case, \(m=4k\) for some \(k \in \IN\), so \(Z-X^k=t\), and \(z-x^k\) is a coordinate.
\end{proof}

\subsection{Generalizing Kuroda's result}
%
This section is devoted to proving the following theorem.
\begin{theorem}\label{thm:Kurodageneral}
Let \(n,a,c,l,s \in \IN\) such that \(cl<a\).  Then the embedding \((t^n, t^{an+c}, t+t^{(an+c)s-ln})\) is rectifiable.
\end{theorem}
\begin{remark}Kuroda \cite{Kuroda} proved the special case of \(c=1\) and \(a=2l+m\) for some nonnegative integer \(m\) (in which case the assumption \(cl<a\) is satisfied automatically).
\end{remark}
\begin{proof}[Proof of Theorem \ref{thm:Kurodageneral}]

Let \(\alpha = \left(x,y,z-\frac{y^s}{x^l}\right)\) and \(\beta = \left(x,y-x^{a-cl}(x^lz)^c,z)\right) \)
We compute
\begin{align*}
\phi^*\alpha^*\beta^*(y) &= Y-X^{a-cl}\left(X^lZ-Y^s\right)^c \\
&=t^{an+c}-t^{n(a-cl)}\left(t^{ln}(t+t^{(an+c)s-ln})-t^{(an+c)s}\right)^c \\
&=t^{an+c}-t^{n(a-cl)}\left(t^{ln+1}\right)^c \\
&= t^{an+c}-t^{n(a-cl)+c(ln+1)} \\
&=0
\end{align*}
It is also easy to see that \(\phi^*\alpha^*\beta^*(z)=t\), so we have \(\beta \alpha \phi = (t^n,0,t)\). Now, from Theorem \ref{thm:cleardenominators}, we can produce \(\gamma = (x,y,z+\frac{R(y)}{x^l})\) such that \(\gamma \beta \alpha \in \GA_2(k[x])\), and \(R(y) \in (y^s)k[x,y]\), so that \((\gamma \beta \alpha \phi)^*(z)=(\beta \alpha \phi)^*(z+\frac{R(y)}{x^l})=t\).  Then letting \(f=(\gamma \beta \alpha)^*(z)\), we have \(f\) is a coordinate with \(\phi^*(f)=t\), so \(\phi\) is rectifiable by Lemma \ref{lem:coordinate}.
\end{proof}

Letting \(c=1,a=3,l=2\), we obtain the following:
\begin{corollary}For any \(n \in \IN\), the embedding
\(\left(t^n, t^{3n+1}, t^{4n+2}+t\right)\) is rectifiable.
\end{corollary}

Letting \(c=2,a=5,l=2\), we obtain the following:
\begin{corollary}For any \(n \in \IN\), the embedding
\(\left(t^n, t^{5n+2}, t^{8n+4}+t\right)\) is rectifiable.
\end{corollary}

\end{document}